\documentclass[12pt,twoside,doublespace]{article}

\usepackage{fancyhdr}
\usepackage{amsmath}
\usepackage{amssymb}
\usepackage{amsfonts}
\usepackage{euscript}
\usepackage{oldgerm}
\usepackage{calligra}
\usepackage{mathrsfs}
\usepackage{hyperref}
\usepackage{url}
\usepackage{lastpage}
\usepackage{multirow}
\usepackage{graphicx}
\usepackage{stmaryrd}
\usepackage{wasysym}
\usepackage{pifont}
\usepackage[geometry]{ifsym2}

\usepackage{draftwatermark}
\SetWatermarkAngle{45}
\SetWatermarkLightness{0.9}
\SetWatermarkFontSize{5cm}
\SetWatermarkScale{2}
\SetWatermarkText{}

\renewcommand{\thefootnote}{\fnsymbol{footnote}}

\long\def\sfootnote[#1]#2{\begingroup
\def\thefootnote{\fnsymbol{footnote}}\footnote[#1]{#2}\endgroup}

\newtheorem{theorem}{Theorem}[section]

\newtheorem{corollary}[theorem]{Corollary}

\newenvironment{proof}{\noindent\mbox{\bf Proof.}}
{\hfill\mbox{$\boxtimes\!\!\!\!\boxtimes$}\bigskip}

\begin{document}
\pagestyle{fancy}
\lhead[page \thepage \ (of \pageref{LastPage})]{{\sf Hazhir Homei}}
\chead[{\bf  Uniform Random Sample and Symmetric Beta Distribution}]{{\bf  Uniform Random Sample and Symmetric Beta Distribution}}
\rhead[{\sf Hazhir Homei}]{page \thepage \ (of \pageref{LastPage})}
\lfoot[]{}
\cfoot[]{}
\rfoot[]{}
\renewcommand{\headrulewidth}{0.2pt}
\renewcommand{\footrulewidth}{0pt}
\thispagestyle{empty}

\begin{table}
\begin{center}
\begin{tabular}{| c | l |}
\hline
 \multirow{9}{*}{\includegraphics[scale=0.75]{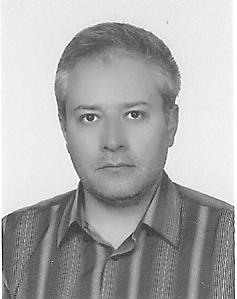}}&    \\
 &   \ \ {\large{\sc Hazhir Homei}}  \ \   \\
 &   \ \ Department of Mathematics \ \  \\
 &     \ \ University of Tabriz \ \ \\
 &    \ \ P.O.Box 51666--17766 \ \ \\
 &   \ \ Tabriz, Iran \ \  \\
 &   \ \    Tel: \, +98 (0)411 339 2863  \ \   \\
  &   \ \ Fax: \ +98 (0)411 334 2102 \ \    \\
   &   \ \ E-mail: {\tt homei}@{\tt tabrizu.ac.ir} \ \    \\
   & \\ 
 \hline
\end{tabular}
\end{center}
\end{table}

\vspace{2em}

\begin{center}
{\bf {\Large Uniform Random Sample and Symmetric Beta Distribution
}}
\end{center}

\vspace{2em}

\begin{abstract}
N.L. Johnson and S. Kotz introduced in 1990 an interesting family of
symmetric distributions which is based on randomly weighted average from
uniform random samples.
The only example that
could be addressed to their work is the so-called ``uniformly
randomly modified tin" distribution from which two random samples
have been computed.
 In this paper, we generalize a subfamily of their symmetric distributions and
 identify a concrete instance of this generalized subfamily. That instance turns out to belong to the family 
 of Johnson and  Kotz,  which had not seemingly received proper attention in the literature.

\bigskip

%{\footnotesize
%\noindent {\bf Acknowledgements} \
%This research is partially supported by a grant from the National Elite Foundation
%(\texttt{Bonyad Melli Nokhbegan} -- \url{www.bmn.ir}) Iran.
%}

\centerline{${\backsim\!\backsim\!\backsim\!\backsim\!\backsim\!\backsim\!\backsim\!
\backsim\!\backsim\!\backsim\!\backsim\!\backsim\!\backsim\!\backsim\!
\backsim\!\backsim\!\backsim\!\backsim\!\backsim\!\backsim\!\backsim\!
\backsim\!\backsim\!\backsim\!\backsim\!\backsim\!\backsim\!\backsim\!
\backsim\!\backsim\!\backsim\!\backsim\!\backsim\!\backsim\!\backsim\!
\backsim\!\backsim\!\backsim\!\backsim\!\backsim\!\backsim\!\backsim\!
\backsim\!\backsim\!\backsim\!\backsim\!\backsim\!\backsim\!\backsim\!
\backsim\!\backsim\!\backsim\!\backsim\!\backsim\!\backsim\!\backsim\!
\backsim\!\backsim\!\backsim\!\backsim\!\backsim\!\backsim\!\backsim\!
\backsim\!\backsim\!\backsim\!\backsim\!\backsim\!\backsim\!\backsim\!
\backsim\!\backsim\!\backsim}$}

\bigskip


\end{abstract}

\bigskip

\bigskip

\vfill

\hspace{.75em} \fbox{\textsl{\footnotesize Date: 11    September  2013  (11.09.13)}}

\vfill

\bigskip
\noindent\underline{\centerline{}}
\centerline{page 1 (of \pageref{LastPage})}

%%%
%%% the paper begins ...
%%%

\newpage
\setcounter{page}{2}
\SetWatermarkAngle{55}
\SetWatermarkLightness{0.955}
%\SetWatermarkFontSize{30cm}
\SetWatermarkScale{2.9}
\SetWatermarkText{\!\!\!\!\!\!\! \copyright\ {\sc Hazhir Homei 2013}}

%%%%%%%%%%%%%%%%%%%%%%%%%%%%%%%%%%%%
%%%%%%%%%%%%%%%%%%%%%%%%%%%%%%%%%%%%
%%%%%%%%%%%%%%%%%%%%%%%%%%%%%%%%%%%%
%%%%%%%%%%%%%%%%%%%%%%%%%%%%%%%%%%%%
%%%%%%%%%%%%%%%%%%%%%%%%%%%%%%%%%%%%

%%%%%%%%%%%%%%%%%%%%%%%%%%%%%%%%%%%%
%%%%%%%%%%%%%%%%%%%%%%%%%%%%%%%%%%%%
%%%%%%%%%%%%%%%%%%%%%%%%%%%%%%%%%%%%
%%%%%%%%%%%%%%%%%%%%%%%%%%%%%%%%%%%%
%%%%%%%%%%%%%%%%%%%%%%%%%%%%%%%%%%%%
%--------------------------------------

\section{Introduction}\label{intro}
 The randomly weighted average, and its applications, has been introduced
by N.L. Johnson and S. Kotz for the first time in 1990 as follows:
for independent random variables $X_1,\cdots,X_n$, the random
variable $Z=\sum_{i=1}^{n}W_iX_i$, where $W=\langle
W_1,\ldots,W_n\rangle$ is independent from $X_i$'s and has the
Dirichlet distribution (with $\sum_iW_i=1$), is called the randomly
weighted average of $X_i$'s. Johnson and Kotz (1990) in the
concluding remarks of their paper  mention that the ``case in which
the $X$'s have standard uniform $(0,1)$ distribution leads to an
interesting family of symmetric distributions."  In this paper, a new representation of Beta
distribution is introduced using randomly weighted average of $n$
 independent Beta random variables.
 This family seems not to have received proper
attention in the literature; though we identify a distribution of it which turns out to be the
celebrated Beta distribution.
 Also, it is shown
that the proposed (moments) method of Johnson and Kotz (1990)
for finding distribution of randomly weighted average instead of the
Stieltjes transform (employed first by Van Assche 1987), which has been
strongly recommended by some authors,
still seems to be a simpler and superior method.

\section{A Presentation for Beta Distribution}

The following theorem provides a new presentation for the Beta
distribution.

\begin{theorem}\label{mainthm}
For independent random variables $X_{1},\cdots,X_{r}$ with
  $Beta(n_{1},m_{1}), \cdots, Beta(n_r,m_r)$ distributions and random
   vector $W=\langle W_1,\cdots,W_r\rangle$ with   $Dirichlet(n_1+m_1,\cdots,n_r+m_r)$  distribution,
 the randomly
weighted average $Z=\sum_{j=1}^{r} W_{j}X_{j}$ has
$Beta(\sum_{j=1}^r n_{j},\sum_{j=1}^r m_{j})$ distribution.
\end{theorem}
\begin{proof}
We find the $k^{\rm th}$ moment of $Z$ as follows:
$$E({Z}^k)=\sum_{i_1+\cdots+i_r=k}^{}\frac{k!}{{i_1}!\cdots {i_r}!}
E({W_1}^{i_1}\cdots {W_r}^{i_r})E({X_1}^{i_1})\cdots
E({X_r}^{i_r}).$$ By using the dirichlet distribution, we have
$$E(Z^k)=\sum_{i_1+\cdots+i_r=k}^{}\frac{k!}{{i_1}!\cdots {i_r}!}
{\frac{\Gamma(\sum_{j=1}^{r}(n_{j}+m_{j}))\prod_{j=1}^{r}\Gamma(n_{j}+m_{j}+i_{j})}
{\prod_{j=1}^{r}\Gamma(n_{j}+m_{j})\Gamma(\sum_{j=1}^{r}(n_{j}+m_{j}+i_{j}))}}E({X_1}^{i_1})\cdots
E({X_r}^{i_r}).$$ It is well known that
$$E({X_j}^{i_j})=\frac{\Gamma(n_{j}+m_{j})\Gamma(i_{j}+n_{j})}{\Gamma(n_{j})\Gamma(i_{j}+m_{j}+n_{j})},
\textrm{ for } j=1,\cdots,r.$$ So, $E(Z^k)=$
$$\sum_{i_1+\cdots+i_r=k}^{}\frac{k!}{{i_1}!\cdots {i_r}!}
{\frac{\Gamma(\sum_{j=1}^{r}(n_{j}+m_{j}))\prod_{j=1}^{r}\Gamma(n_{j}+m_{j}+i_{j})}
{\prod_{j=1}^{r}\Gamma(n_{j}+m_{j})\Gamma(\sum_{j=1}^{r}(n_{j}+m_{j}+i_{j}))}}
\prod_{j=1}^{r}\frac{\Gamma(n_{j}+m_{j})\Gamma(i_{j}+n_{j})}{\Gamma(n_{j})\Gamma(i_{j}+m_{j}+n_{j})}$$
$$=\frac{\Gamma(\sum_{j=1}^{r}(n_{j}+m_{j}))}{\Gamma(\sum_{j=1}^{r}(n_{j}+m_{j}+i_{j}))}
\sum_{i_1+\cdots+i_r=k}^{}\frac{k!}{{i_1}!\cdots {i_r}!}
\prod_{j=1}^{r}\frac{\Gamma(i_j+n_j)}{\Gamma(n_j)}$$
$$=\frac{\Gamma(\sum_{j=1}^{r}(n_{j}+m_{j}))}{\Gamma(\sum_{j=1}^{r}(n_{j}+m_{j}+i_{j}))}
\sum_{i_1+\cdots+i_r=k}^{}{k\choose i_{1}, i_{2}, \cdots,
i_{r}}\frac{\Gamma(i_1+n_1)}{\Gamma(n_1)}\cdots\frac{\Gamma(i_r+n_r)}{\Gamma(n_r)}.$$
By considering the fact that the sum of the Dirichlet-multimonial
distribution on its support equals to one, we have
$$E(Z^k)=\frac{\Gamma(\sum_{j=1}^{r}(n_{j}+m_{j}))}{\Gamma(\sum_{j=1}^{r}(n_{j}+m_{j}+i_{j}))}
\frac{\Gamma(k+\sum_{j=1}^{r} n_j)}{\Gamma(\sum_{j=1}^{r} n_j)},$$
which is the $k^{\rm th}$ moment of the $Beta(\sum_{j=1}^r
n_{j},\sum_{j=1}^r m_{j})$ distribution, and this proves the
theorem.
\end{proof}

The following corollary  
 exactly specifies (only) one instance of the interesting family introduced by Johnson and Kotz (1990).

\begin{corollary}\label{coro1}
For independent random variables $X_{1},\cdots,X_{r}$ with standard
uniform $(0,1)$ distributions and random vector $W=\langle
W_1,\cdots,W_r\rangle$ with $Dirichlet(2,\cdots,2)$ distribution,
the randomly weighted average $Z=\sum_{j=1}^{r} W_{j}X_{j}$ has
$Beta(r,r)$ distribution.
\end{corollary}

\begin{proof}
In Theorem~\ref{mainthm} put $n_j=m_j=1$ for $j=1,\cdots,r$.
\end{proof}

In some papers such as Roozegar and Soltani (2013) it is mentioned that the ``distribution of
[randomly weighted average] cannot be directly specified", and ``[i]t is realized that
Stieltjes transform (ST) is more appropriate".
In fact, the following corollary of Theorem~\ref{mainthm} contains the main result of Roozegar and Soltani (2013),
noting that the $R_i$'s defined by them as $R_{i}=U_{(i)}-U_{(i-1)}$ and $R_n= 1-\sum_{i=1}^{n-1} R_i$ where
$U_{(1)},...,U_{(n-1)}$ are order statistics of a random sample
$U_1,...,U_n$ from a uniform
distribution on [0,1] with $U_{(0)}=0$ and $U_{(n)}=1$, have  $Dirichlet(1,\cdots,1)$ distribution.

\begin{corollary}\label{coro1}
For independent random variables $X_1,\cdots,X_r$ with 
common Arcsin distribution on $(-a,a)$ and random vector $W=\langle
W_1,\cdots,W_r\rangle$ with $Dirichlet(1,\cdots,1)$ distribution,
the randomly weighted average $Z=\sum_{j=1}^{r} W_{j}X_{j}$ has a
power semicircle distribution on $(-a,a)$ with
$\lambda=\frac{r-1}{2}$; i.e., 
$$f(x;\lambda)=(\frac{1}{\sqrt \pi a^{2\lambda}})\frac{\Gamma(\lambda+1)}
{\Gamma(\lambda+\frac{1}{2})}(a^2-x^2)^{\lambda-\frac{1}{2}}\quad \text{\rm for } \;|x|<a.$$
\end{corollary}
\begin{proof}
In Theorem~\ref{mainthm} put $n_j=m_j=\frac{1}{2}$ for
$j=1,\cdots,r$. Clearly, $2aZ-a=\sum_{j=1}^{r}
W_{j}(2aX_{j}-a)$ and $(2aX_j-a)$'s  have arcsin distributions.
\end{proof}

\section{Conclusions}
The method of this article, which is a continuation of the method of
Johnson and Kotz (1990), provides an elementary and direct way for
computing the distribution of randomly weighted averages of certain
random variables. Our Theorem~\ref{mainthm} provided one specific
example of the interesting family of symmetric distributions
foreseen in Johnson and Kotz (1990). 
 So, this goes to say that the more elementary and much simpler method of Johnson
and Kotz (1990) can still be a powerful technique.

\end{document}